\documentclass{amsart}
\usepackage[T1]{fontenc}
\usepackage[utf8]{inputenc}

\usepackage[noadjust]{cite}
\usepackage{palatino}
\usepackage{hyperref}
\usepackage{amssymb,amsthm,units,mathabx}
\usepackage{color}
\usepackage[usenames,dvipsnames,svgnames,table]{xcolor}
\usepackage{mathtools,xparse,xspace}
\usepackage{rotating}
\usepackage{bm} 

\hypersetup{allbordercolors=red,linkcolor=red,urlcolor=blue,urlbordercolor=blue,colorlinks=false,pdfborder={0 0 1}} 

\usepackage{epsfig}  		
\usepackage{epic,eepic}       
\usepackage{tikz,pgf}
\usetikzlibrary{arrows}
\usepackage[abs]{overpic}

\newtheorem{thm}{Theorem}[section]
\newtheorem{prop}[thm]{Proposition}
\newtheorem{lem}[thm]{Lemma}

\newtheorem{question}[thm]{Question}

\theoremstyle{definition}

\theoremstyle{remark}

\newtheorem{remark}[thm]{Remark}

\numberwithin{equation}{section}

\newcommand{\R}{\mathbb{R}}  
\newcommand{\bd}{\partial}  

\newcommand{\Mxi}{(M, \xi)} 
\newcommand{\LOSS}{\widehat{\mathcal{L}}} 
\newcommand{\HFhat}{\widehat{{HF}}} 
\newcommand{\HFK}{\widehat{\mathrm{HFK}}} 
\newcommand{\SFH}{SFH} 

\newcommand{\vects}[2]{\left(\begin{smallmatrix} #1 \\ #2 \end{smallmatrix}\right)} 

\begin{document}

\title{Contact Surgeries on the Legendrian Figure-Eight Knot}
\author{James Conway}
\address{University of California, Berkeley}
\email{conway@berkeley.edu}
\begin{abstract}
We show that all positive contact surgeries on every Legendrian figure-eight knot in $(S^3, \xi_{\rm{std}})$ result in an overtwisted contact structure. The proof uses convex surface theory and invariants from Heegaard Floer homology.
\end{abstract}
\maketitle

\section{Introduction}

Dehn surgery on knots has been a fruitful way to construct new contact structures on 3-manifolds, and in particular to try to construct new tight contact manifolds.  When the knot in question is a Legendrian knot (\textit{ie.\@} its tangent vectors lie in the contact planes), Dehn surgery with framing equal to the contact framing always results in an overtwisted overtwisted contact manifold.  The remaining framings break into two classes: those less than the contact framing, and those greater.  Surgeries with these framings give rise to negative and positive contact surgery, respectively.

In \cite{Wand}, Wand showed that given a tight contact manifold, the result of negative contact surgery on any Legendrian knot is a tight contact manifold.  Regarding positive contact surgery, much less is known; existing tightness results can be found in \cite{LS:hf, LS:hf1, LS:seifert2, LS:surgery, Conway:transverse, MT, Golla}.

Most of the results for positive contact surgery prove tightness using various flavours of Heegaard Floer homology.  In particular, the non-vanishing of the Heegaard Floer contact class shows that a contact manifold is tight, however its vanishing is not equivalent to a contact manifold being overtwisted.  Several of the above results give conditions under which contact $(+1)$-surgery (\textit{ie.\@} positive contact surgery with framing one more than the contact framing) has vanishing Heegaard Floer contact class.

The results that prove that a positive contact surgery in overtwisted are even fewer.  Lisca and Stipsicz showed in \cite{LS:hfnotes} that there exists a configuration in the front projection of a Legendrian knot that ensures contact $(+1)$-surgery on the knot is overtwisted.  This configuration is not present in the figure-eight knot under consideration in this paper (but it is present in the negative torus knots, for example).  In \cite{Conway:transverse}, the author used versions of the Bennequin inequality (an inequality of Legendrian knot invariants that holds in tight contact manifolds) to give general results for when positive contact surgery on Legendrian knots is overtwisted.

After the unknot and the trefoils, the figure-eight knot is next natural knot to study (contact surgeries on the others were understood by \cite{LS:hf,DGS:surgery} for the unknot, \cite{LS:hf1} for the right-handed trefoil, and \cite{LS:hfnotes,Conway:transverse} for the left-handed trefoil).  The classification of Legendrian figure-eight knots in $(S^3, \xi_{\rm{std}})$ was undertaken by Etnyre and Honda in \cite{EH:knots}, who proved that all such Legendrian knots are classified up to isotopy by their Thurston--Bennequin number ($tb$) and rotation class ($rot$), and that all such knots destabilise to a Legendrian knot with $tb = -3$ and $rot = 0$.  Lisca and Stipsicz showed in \cite{LS:hfnotes} that the result of contact $(+1)$-surgery on any Legendrian figure-eight knot has vanishing Heegaard Floer contact class; we answer the natural follow-up question:

\begin{thm}\label{maintheorem}
The results of all positive contact surgeries on any Legendrian figure-eight knot in $(S^3, \xi_{\rm{std}})$ are overtwisted.
\end{thm}

\begin{remark}
One should not conclude from Theorem~\ref{maintheorem} that the manifolds resulting from surgery on the figure-eight support no tight contact structure: in fact, they all support tight contact structures.  However, they do not arise from positive contact surgery on a figure-eight knot in $(S^3, \xi_{\rm{std}})$.
\end{remark}

The proof uses convex surfaces and the Heegaard Floer contact class.  In particular, given any Legendrian knot $L$ we show that if any positive contact surgery on $L$ is tight, then a particular contact structure on $S^3 \backslash N(K)$ is also tight. For the figure-eight knot, we can show that this contact structure $\xi$ on $S^3 \backslash N(K)$ has vanishing Heegaard Floer contact class. We then use convex surfaces to classify all tight contact structures on $S^3 \backslash N(K)$ that induce a particular set of dividing curves on a convex Seifert surface (the same set of curves can also be found in $\xi$).  We then construct these tight contact structures, and show that they have non-vanishing Heegaard Floer contact class.  This shows that $\xi$ is overtwisted, and proves Theorem~\ref{maintheorem}.

Beyond the figure-eight knot, it is unclear how successful this approach will be.  The fact that the figure-eight knot is fibred and genus 1 play a large role in making the classification of relevant tight contact structures on $S^3 \backslash N(K)$ possible.  However, the approach of showing that a particular contact structure on $S^3 \backslash N(K)$ is overtwisted is more widely applicable, as can be seen in \cite{Conway:transverse}.

In all known cases where the result of positive contact surgery on a Legendrian knot $(S^3, \xi_{\rm{std}})$ is tight, we also know that the Heegaard Floer contact invariant is non-vanishing.  This paper, along with the results in \cite{Conway:transverse}, lend support toward a positive answer to this question:

\begin{question}
Let $\Mxi$ be the result of some positive contact surgery on a Legendrian knot in $(S^3, \xi_{\rm{std}})$.  Is $\xi$ tight if and only if its Heegaard Floer contact class is non-vanishing?
\end{question}

\subsection*{Acknowledgements}

The author is indebted to John Etnyre for many helpful discussions.  The results in this paper originally appeared as part of the author's doctoral thesis.  This work was partially supported by NSF Grant DMS-13909073.

\section{Contact Geometric Background}
\label{sec:background}

We begin with a brief reminder of standard theorems about contact structures on $3$-manifolds which we will use throughout this paper.  We assume a basic knowledge of contact structures at the level of \cite{Etnyre:contactlectures,Etnyre:contactlectures1}.

\subsection{Farey Graph}\label{sec:farey}
The Farey graph is the $1$-skeleton of a tessellation of the hyperbolic plane by geodesic triangles shown in Figure~\ref{fig:farey}, where the endpoints of the geodesics are labeled.  The labeling, shown in Figure~\ref{fig:farey}, is determined as follows: let the left-most point be labeled $\infty = 1/0$ and the right-most point be labeled $0$.  Given a geodesic triangle where two corners are already labeled $a/b$ and $c/d$, then the third corner is labeled $(a+c)/(b+d)$.  For triangles in the upper half of the plane, we treat $0$ as $0/(-1)$, whereas for triangles in the lower half of the plane, we treat $0$ as $0/1$.  Thus, the labels on the upper half are all negative, and those on the lower half are all positive.  Every rational number and infinity is found exactly once as a label on the Farey graph.

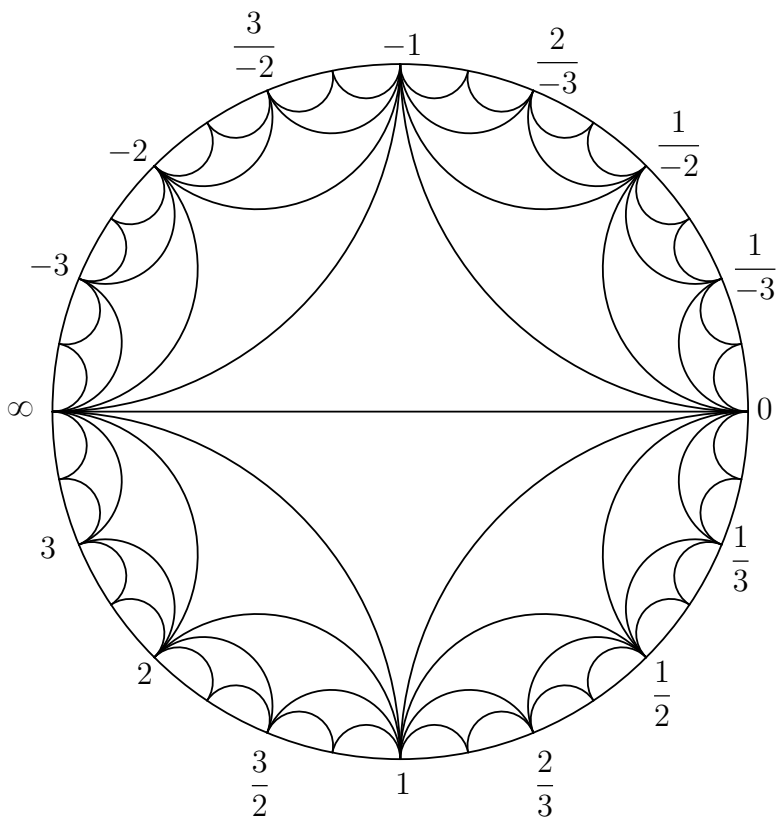
\begin{figure}[htbp]
\begin{center}
\vspace{1cm}
\scalebox{.85}{\begin{overpic}[scale=1.7,tics=20]{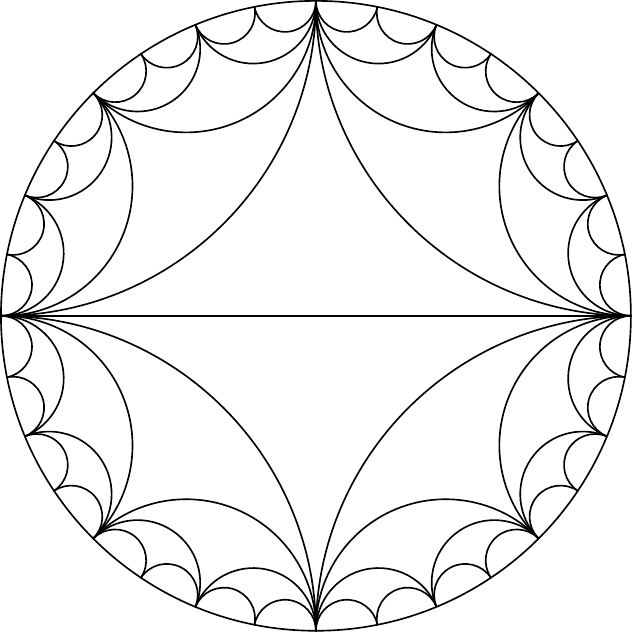}
\put(-20,153){\LARGE $\infty$}
\put(314,152){\LARGE $0$}
\put(153,-15){\LARGE $1$}
\put(147,314){\LARGE $-1$}
\put(25,267){\LARGE $-2$}
\put(-10,216){\LARGE $-3$}
\put(80,314){\LARGE $\displaystyle\frac{3}{-2}$}
\put(269,270){\LARGE $\displaystyle\frac{1}{-2}$}
\put(303,215){\LARGE $\displaystyle\frac{1}{-3}$}
\put(215,307){\LARGE $\displaystyle\frac{2}{-3}$}
\put(267,25){\LARGE $\displaystyle\frac{1}{2}$}
\put(302,85){\LARGE $\displaystyle\frac{1}{3}$}
\put(215,-15){\LARGE $\displaystyle\frac{2}{3}$}
\put(88,-15){\LARGE $\displaystyle\frac{3}{2}$}
\put(38,34){\LARGE $2$}
\put(-5,90){\LARGE $3$}
\end{overpic}}
\vspace{1cm}
\caption[The Farey graph]{The Farey graph.}
\label{fig:farey}
\end{center}
\end{figure}

\subsection{Convex Surfaces}

We introduce the basics of convex surfaces.  See \cite{Etnyre:convex} for more details.

A surface $\Sigma$ (possibly with boundary) in a contact manifold $\Mxi$ is called {\it convex} if there exists a {\it contact vector field} $v$ such that $v$ is transverse to $\Sigma$.  Here, a {\it contact vector field} is a vector field whose flow preserves the contact planes.  Using the contact vector field $v$, it is not hard to see that convex surfaces have a neighbourhood contactomorphic to $\Sigma \times \R$ with an $\R$-invariant contact structure, called a {\it vertically-invariant neighbourhood} of $\Sigma$.

Given a surface $\Sigma$ in $\Mxi$ and the characteristic foliation $\mathcal{F}$ on $\Sigma$ induced by $\xi$, we say that a multi-curve $\Gamma$ on $\Sigma$ {\it divides} $\mathcal{F}$ if

\begin{itemize}
\item $\Sigma \backslash \Gamma = \Sigma_+ \sqcup \Sigma_-$,
\item $\Gamma$ is transverse to the singular foliation $\mathcal{F}$, and
\item there is a volume form $\omega$ on $\Sigma$ and a vector field $w$ such that
\begin{itemize}
\item $\pm \mathcal{L}_w \omega > 0$ on $\Sigma_{\pm}$,
\item $w$ directs $\mathcal{F}$, and
\item $w$ points out of $\Sigma_+$ along $\Gamma$.
\end{itemize}
\end{itemize}

\begin{thm}[Giroux \cite{Giroux:convex}]\label{convex generic} 
%

A closed surface $\Sigma$ is $C^\infty$-close to a convex surface.  If $\Sigma$ is a surface with Legendrian boundary such that the twisting of the contact planes along each boundary component is non-positive when measured against the framing given by $\Sigma$, then $\Sigma$ can be $C^0$-perturbed in a neighbourhood of the boundary and $C^\infty$-perturbed on its interior to be convex.

If $\Sigma \subset \Mxi$ is an orientable surface, and its boundary (if it is non-empty) is Legendrian, then $\Sigma$ is a convex surface if and only if its characteristic foliation has a dividing set.  Given a convex surface $\Sigma$ with dividing curves $\Gamma$, and any singular foliation $\mathcal{F}$ on $\Sigma$ divided by $\Gamma$, then $\Sigma$ can be perturbed to a convex surface with characteristic foliation $\mathcal{F}$. \end{thm}

In particular, convex surfaces are generic, and the germ of the contact structure at a convex surface is determined (up to a $C^0$-perturbation of the surface) by its dividing curves and the signs of the regions $\Sigma_\pm$.

A properly-embedded graph $G$ on a convex surface $\Sigma$ is {\it non-isolating} if $G$ intersects the dividing curves $\Gamma$ transversely, and each component of $\Sigma \backslash G$ has non-trivial intersection with $\Gamma$.

\begin{thm}[Honda \cite{Honda:classification1}]\label{LRP} If $G$ is a non-isolating properly-embedded graph on a convex surface $\Sigma$, then there is an isotopy of $\Sigma$ relative to its boundary such that $G$ is contained in the new characteristic foliation.  If $G$ is a simple closed curve, then the twisting of the contact planes along $L$ with respect to the framing on $G$ given by $\Sigma$ is equal to $$tw(G,\Sigma) = -\frac{|G \cap \Gamma|}{2}.$$\end{thm}

This is commonly called the {\it Legendrian realisation principle}.  In particular, a simple closed curve in $\Sigma$ that is non-separating can always be Legendrian realised on a convex surface.  If $L$ is a null-homologous Legendrian knot bounding a convex surface, then $tw(L,\Sigma) = tb(L)$, and so $tb(L) = -|L \cap \Gamma|/2$.

Giroux has shown that there are restrictions on dividing curves in tight manifolds.  This result is often called {\it Giroux's Criterion}.

\begin{thm}[Giroux \cite{Giroux:convex}]\label{contractible} If $\Sigma = S^2$ is convex, then a vertically-invariant neighbourhood of $\Sigma$ is tight if and only if the dividing set $\Gamma$ is connected.  If $\Sigma \neq S^2$, then a vertically-invariant neighbourhood of $\Sigma$ is tight if and only if $\Gamma$ has no contractible components.\end{thm}

Given two convex surfaces $\Sigma_1$ and $\Sigma_2$ that intersect in a Legendrian curve $L$, Kanda \cite{Kanda} and Honda \cite{Honda:classification1} have shown that between each intersection of $L$ with $\Gamma_{\Sigma_1}$ is exactly one intersection of $L$ with $\Gamma_{\Sigma_2}$, as in Figure~\ref{fig:two convex surfaces}.  Honda further showed that there is a way to ``round edges'' at $L$ and get a new convex surface.  The dividing set on the new surface is derived from $\Gamma_{\Sigma_i}$ as in Figure~\ref{fig:rounding convex boundary}.

\begin{figure}[htbp]
\begin{center}
\scalebox{.95}{\begin{overpic}[scale=1,tics=20]{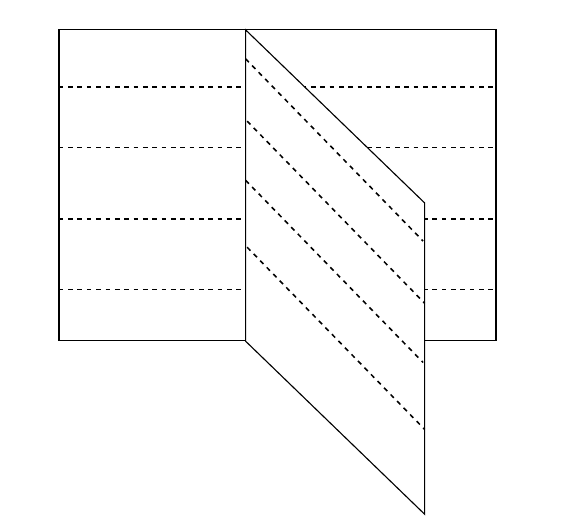}
\put(40,40){$\Sigma_1$}
\put(80,20){$\Sigma_2$}
\end{overpic}}
\caption[Intersecting Convex Surfaces]{Two convex surfaces intersecting in a Legendrian curve.  This figure is reproduced from \cite[Figure 10]{Etnyre:convex}.}
\label{fig:two convex surfaces}
\end{center}
\end{figure}

\begin{figure}[htbp]
\begin{center}
\scalebox{.95}{\begin{overpic}[scale=1,tics=20]{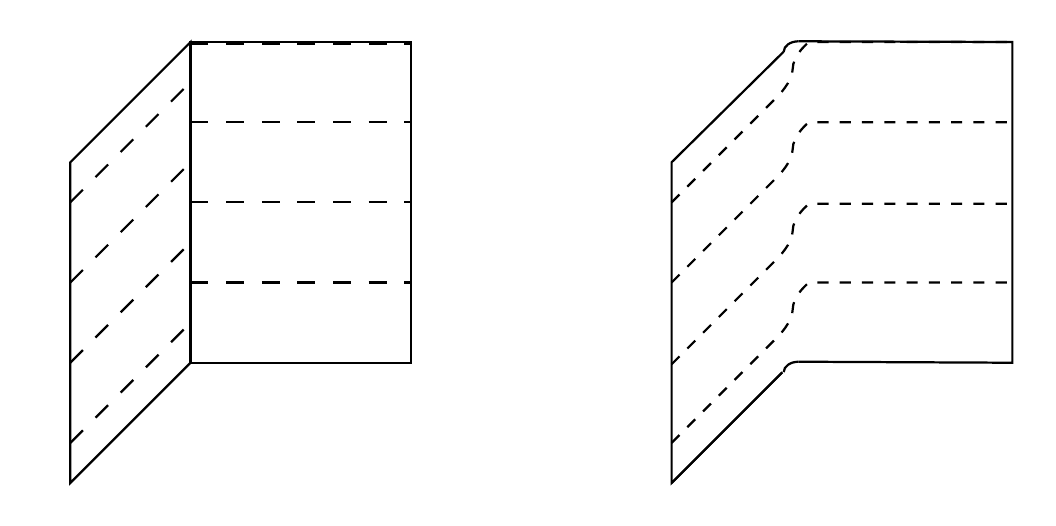}
\put(40,15){$\Sigma_1$}
\put(80,37){$\Sigma_2$}
\end{overpic}}
\caption[Rounding convex surfaces]{``Rounding edges'' of intersecting convex surfaces. This figure is reproduced from \cite[Figure 11]{Etnyre:convex}.}
\label{fig:rounding convex boundary}
\end{center}
\end{figure}

A special case of ``rounding edges'' at the intersection of two convex surfaces is when $\Sigma_2$ is a {\it bypass}.  This is when $\Sigma_2$ is a disc with Legendrian boundary with $tb = -1$, such that $\Sigma_1 \cap \Sigma_2$ is an arc $\alpha$ intersecting $\Gamma_{\Sigma_1}$ in three points, two of which are the endpoints of $\alpha$; we further require that the endpoints of $\alpha$ are elliptic singularities of the characteristic foliation on $\Sigma_2$.  By the above discussion, the dividing set $\Gamma_{\Sigma_2}$ is a single arc with endpoints on $\alpha$.  By Theorem~\ref{convex generic}, we can arrange for there to be a unique hyperbolic singularity on $\bd\Sigma_2$ that lies on $\alpha$ and is between the two points $\alpha \cap \Gamma_{\Sigma_2}$.  The sign of this hyperbolic singularity is called the {\it sign} of the bypass.

Honda proved \cite{Honda:classification1} that in a neighbourhood of $\Sigma_1 \cup \Sigma_2$, there is a one-sided neighbourhood $\Sigma_1 \times [0, 1]$ of $\Sigma_1$ such that $\Sigma_1 \times \{0, 1\}$ is convex, the dividing curves on $\Sigma_1 \times \{0\}$ are $\Gamma_{\Sigma_1}$, and the dividing curves on $\Sigma_1 \times \{1\}$ are $\Gamma_{\Sigma_1}$ changed along a neighbourhood of $\alpha$ as in Figure~\ref{fig:bypass}.  We say that the convex surface $\Sigma_1 \times \{1\}$ is obtained from $\Sigma_1$ by a bypass attachment along $\Sigma_2$.

\begin{figure}[htbp]
\begin{center}
\vspace{0.2cm}
\begin{overpic}[scale=1.5,tics=20]{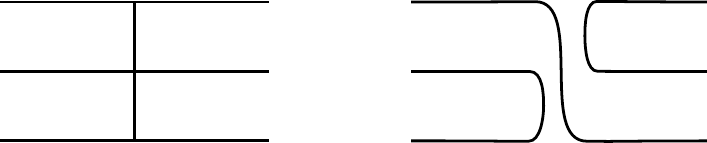}
\put(62,15){$\alpha$}
\put(128,31){\line(1,0){40}}
\put(160.5,28.3){$>$}
\end{overpic}
\caption[Bypass attachment]{The result of performing a bypass on the dividing curves.}
\label{fig:bypass}
\end{center}
\end{figure}

If $\Sigma_1$ is a convex $T^2$ (resp.\ $T^2\backslash D^2$) with 2 parallel dividing curves, then we can choose the characteristic foliation on $\Sigma_1$ such that it consists of two curves called {\it Legendrian divides} parallel to the dividing curves along with a linear foliation of the torus by curves not parallel to the dividing curves, called {\it ruling curves}.  Under these hypotheses, Honda proved \cite{Honda:classification1} how the slopes of the dividing curves change under bypass attachments along a ruling curve.  Denote the slope of curves parallel to $\vects{q}{p}$ by $p/q$, as in the Farey graph.

\begin{thm}[Honda \cite{Honda:classification1}]\label{bypass Farey graph} Let $\Sigma_1$ have two dividing curves of slope $s$ and ruling curves of slope $r$.  Let $\Sigma_2$ be a bypass attached to $\Sigma_1$ along a ruling curve.  Then the result $\Sigma'_1$ of a bypass attachment along $\Sigma_2$ has two dividing curves with slope $s'$, where $s'$ is the label on the Farey graph clockwise of $r$ and counter-clockwise of $s$, and such that $s'$ is the label closest to $r$ with an edge to $s$. \end{thm}

\begin{remark} If $\Sigma_2$ is a bypass for $\Sigma_1$ attached along the back of $\Sigma_1$, then the bypass attachment will change $\Gamma_{\Sigma_1}$ in a manner similar to Figure~\ref{fig:bypass} but reflected in the vertical axis. Theorem~\ref{bypass Farey graph} will hold after reversing the words ``clockwise'' and ``counter-clockwise''.\end{remark}

Bypasses are only useful if we can find them.  To that effect, we have the Imbalance Principle, which allows us to find bypasses on annuli.

\begin{thm}[Honda \cite{Honda:classification1}]\label{imbalance principle} Let $\Sigma$ and $A = S^1 \times [0, 1]$ be two convex surfaces with Legendrian boundary, such that $\Sigma \cap A = S^1 \times \{0\}$.  Then, if the twisting of the contact planes along the boundary of $A$ satisfies $$tw(S^1 \times \{0\},A) < tw(S^1 \times \{1\},A) \leq 0,$$ then there is a bypass for $\Sigma$ along $A$, {\it ie.\ }some subsurface of $A$ is a bypass for $\Sigma$.  \end{thm}

In particular, if $S^1 \times \{1\}$ sits on a convex surface $\Sigma'$, and $$\left|\Gamma_{\Sigma} \cap \left(S^1 \times\{0\}\right)\right| > \left|\Gamma_{\Sigma'}\cap\left(S^1 \times\{0\}\right)\right|,$$ then the hypotheses of Theorem~\ref{imbalance principle} hold, and there is a bypass for $\Sigma$ along $A$.

\subsection{Basic Slices}

Consider the manifold $(T^2 \times I,\xi)$, with $\xi$ tight.  Let the two boundary components be convex with two dividing curves each, with slopes $s_0$ and $s_1$.  If $s_0$ and $s_1$ are labels on the Farey graph connected by a geodesic, then $(T^2 \times I,\xi)$ is called a {\it basic slice}.  If not, then the manifold can be cut up into basic slices along boundary parallel convex tori, following the path between $s_0$ and $s_1$ along the Farey graph.

\begin{thm}[Honda \cite{Honda:classification1}]\label{hondaclassification} There are exactly two tight contact structures up to isotopy (and only one up to contactomorphism) on $T^2 \times I$ with a fixed singular foliation on the boundary that is divided by two dividing curves on $T^2 \times \{i\}$ for $i = 0,1$ each of slope $s_i$, where $s_0$ and $s_1$ are labels in the Farey graph connected by a geodesic.
\end{thm}

The two tight contact structures can be distinguished by their relative Euler class, and after picking an orientation, we can call them {\it positive} and {\it negative} basic slices; this orientation is chosen such that when gluing a negative (resp.\ positive) basic slice to the boundary of the complement of a regular neighbourhood of a Legendrian knot, the result is the complement of a regular neighbourhood of its negative (resp.\ positive) stabilisation.

In addition, this classification implies that if we have a basic slice $(T^2 \times I,\xi)$ that can be broken up into two basic slices $(T^2 \times [0, 1/2], \xi_1)$ and $(T^2, \times [1/2, 1], \xi_2)$, then the sign of each of the latter two basic slices agrees with the sign of $(T^2 \times I, \xi)$.  Thus, if the signs disagree, then $(T^2 \times I, \xi)$ is overtwisted (and hence by definition not a basic slice).

\subsection{Contact Surgery}

Given a null-homologous Legendrian knot $L \subset \Mxi$, we start by removing the interior of a standard neighbourhood $N(L)$ of $L$, {\it ie.\ }the interior of a tight solid torus with convex boundary, where the dividing curves have the same slope as the contact framing $tb(L)\mu + \lambda$, where $\mu$ is a meridian and $\lambda$ is the Seifert framing of $L$.

To do \textit{positive contact surgery} on $L$, we first glue a basic slice to $\bd N(L)$ such that the new contact structure on $M \backslash N(L)$ has convex boundary with two meridional dividing curves.  Different sign choices on this basic slice in general give rise to distinct contact structures; we denote by $\xi^+(L)$ (resp.\@ $\xi^-(L)$) the contact structure on $M\backslash N(L)$ coming from gluing on a positive (resp.\@ negative) basic slice. Finally, we then glue a solid torus to the boundary such that the desired topological surgery is achieved, and we extend the contact structure over the solid torus such that it is tight on the solid torus. Different choices of sign on the basic slice and different extensions over the solid torus will in general give rise to distinct contact structures on the surgered manifold, see \cite{Kanda, Honda:classification1}.

\subsection{Heegaard Floer Homology}\label{sec:heegaardfloer}

We make use invariants of contact structures coming from Heegaard Floer theory: for closed contact manifolds $\Mxi$, we have an element $c(\xi) \in \HFhat(-M)$ (see \cite{OS:contact}), and for contact manifolds $(M', \Gamma, \xi')$ with convex boundary, where $\Gamma \subset \bd M'$ is the dividing set, we have an element $EH(\xi) \in \SFH(-M', -\Gamma)$ (see \cite{HKM:sutured}).  If $(M', \Gamma, \xi') \subset \Mxi$ is a contact embedding, then there is a map $\SFH(-M',-\Gamma) \to \HFhat(-M)$ that sends $EH(\xi')$ to $c(\xi)$.

To a Legendrian knot $L \subset \Mxi$, we associate an element $\LOSS(L)$ (defined in \cite{LOSS}) in the knot Heegaard Floer group $\HFK(-M,-L)$.  For knots in $(S^3, \xi_{\rm{std}})$, $\LOSS(L)$ was identified (up to an automorphism of the ambient group) in \cite{BVVV} with a more easily calculable invariant defined in \cite{OST:knots}; this latter invariant can be shown to vanish for any Legendrian figure-eight knot $L$ (as $\HFK(-S^3,-L)$ is trivial in the required grading).  In \cite{SV}, the element $\LOSS(L)$ was also identified with the class $EH(\xi_{\rm{std}}^-(L))$ of $(S^3 \backslash N(K), \xi_{\rm{std}}^-(L))$, under an isomorphism $\HFK(-S^3,-L) \cong \SFH(-S^3 \backslash N(K), -\Gamma_{\rm{meridional}})$.

\section{Surgeries on the Figure-Eight Knot}
\label{sec:figure-eight}

Consider the figure-eight knot $K$ in $S^3$ (see Figure~\ref{figure8knot}).  We will show that the result of any positive contact surgery on any Legendrian realisation of the figure-eight knot in $(S^3, \xi_{\rm{std}})$ is overtwisted.

\begin{figure}[htbp]
\begin{center}
\includegraphics[scale=0.55]{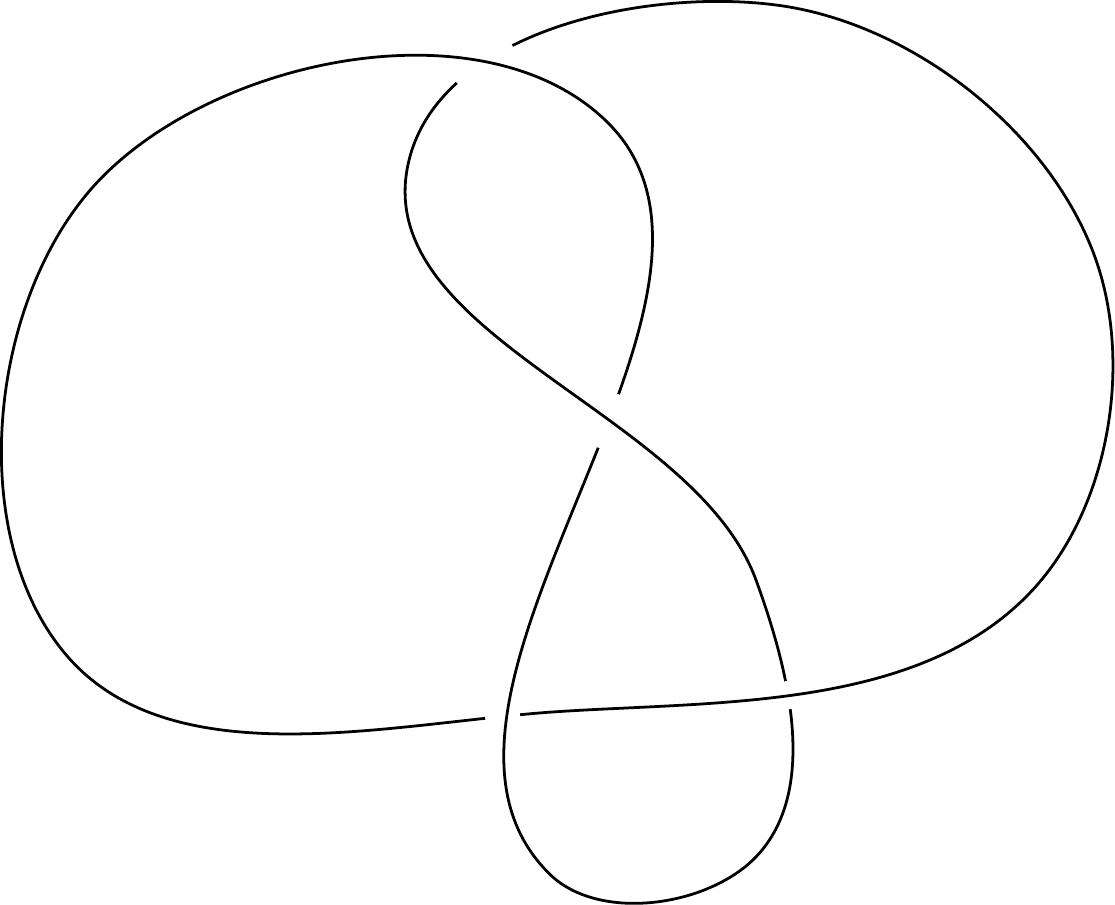} \hfill \includegraphics[scale=0.65]{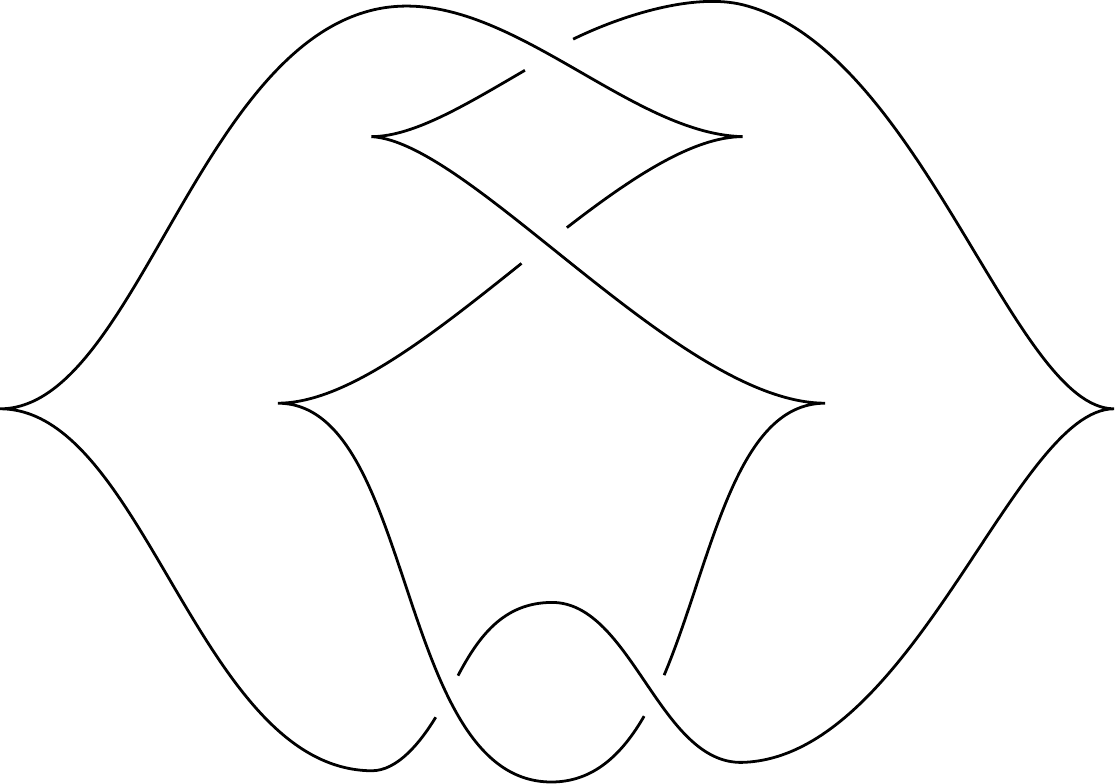}
\caption[Figure-Eight Knot.]{On the left is a smooth figure-eight knot $K$.  On the right is a Legendrian representative $L$ of $K$ with $tb(L) = \overline{tb}(K) = -3$.  We omit choices of orientation, since $K$ is amphichiral.}
\label{figure8knot}
\end{center}
\end{figure}

Let $L$ be a Legendrian figure-eight knot in $(S^3, \xi_{\rm{std}})$. Define a contact structure $\xi^-(L)$ (resp.\ $\xi^+(L)$) on $S^3 \backslash N(K)$ by gluing a negative (resp.\ positive) basic slice to the complement of $N(L) \subset (S^3, \xi_{\rm{std}})$ such that $\bd \left(S^3 \backslash N(K)\right)$ is convex with two meridional dividing curves.

\begin{prop}\label{prop:lessthan-3}
Let $L$ be a Legendrian figure-eight knot in $(S^3, \xi_{\rm{std}})$.
\begin{enumerate}
\item If $tb(L) - rot(L) = -3$ and $tb(L) < -3$, then $(S^3 \backslash N(K), \xi^+(L))$ is overtwisted.
\item If $tb(L) + rot(L) = -3$ and $tb(L) < -3$, then $(S^3 \backslash N(K), \xi^-(L))$ is overtwisted.
\item If $tb(L) \pm rot(L) < -3$, then $(S^3 \backslash N(K), \xi^{\pm}(L))$ is overtwisted.
\end{enumerate}
\end{prop}
\begin{proof}
For any Legendrian knot $L$, $(S^3 \backslash N(K), \xi^-(L))$ is contactomorphic to $(S^3 \backslash N(K), \xi^+(\overline{L}))$, where $\overline{L}$ is the mirror Legendrian knot to $L$. Since the figure-eight knot is amphichiral, $\overline{L}$ is also a figure-eight knot, and $rot(\overline{L}) = -rot(L)$.  Thus, $(1)$ and $(2)$ are equivalent.  Also, if $L$ satisfies $tb(L) \pm rot(L) < -3$, then so does $\overline{L}$, so to prove the proposition, it suffices to consider $\xi^-(L)$ for $L$ satisfying the hypotheses of $(2)$ and $(3)$.

By \cite{EH:knots}, the figure-eight knot is a Legendrian simple knot ({\it ie.\ }Legendrian figure-eight knots are classified up to isotopy by their $tb$ and $rot$) with $tb(L) - rot(L) \leq -3$.  Thus, any such $L$ satisfying the hypotheses of $(2)$ or $(3)$ is a positive stabilisation of some other Legendrian knot $L'$.  Gluing a negative basic slice to the complement of $L$ to construct the contact structure $\xi^-(L)$ is the same as first gluing a positive basic slice to the complement of $L'$ to arrive at the complement of $L$, and then gluing on the negative basic slice to get $\xi^-(L)$.  These two basic slices (the positive and the negative) glue together to give one $T^2 \times I$, but since the two basic slices have opposite signs, the contact structure on this $T^2 \times I$ is overtwisted (see the discussion after Theorem~\ref{hondaclassification}).  This $T^2 \times I$ embeds into $(S^3 \backslash N(K), \xi^-(L))$, so we conclude that $(S^3 \backslash N(K), \xi^-(L))$ is overtwisted.
\end{proof}

Let $L_t$ have $tb(L_t) = t \leq -3$ and $tb(L_t) - rot(L_t) = -3$. The negative basic slice with dividing curve slopes $-3$ and $\infty$ can be divided into two negative basic slices, one with dividing curve slopes $-3$ and $t$, and one with dividing curve slopes $t$ and $\infty$.  Hence, $(S^3 \backslash N(L_t), \xi^-(L_t)) = (S^3 \backslash N(L_{-3}), \xi^-(L_{-3}))$ for all $t \leq -3$.  A similar statement holds for $\xi^+$ for $L$ satisfying $tb(L) + rot(L) = -3$.  Additionally, as in the proof of Proposition~\ref{prop:lessthan-3}, the amphichirality of the figure-eight knot gives a contactomorphism between $\xi^-(L_{-3})$ and $\xi^+(L_{-3})$.  Thus, to prove Theorem~\ref{maintheorem}, it is sufficient to show that $(S^3 \backslash N(K), \xi^-(L_{-3}))$ is overtwisted.

For the rest of this section, let $L$ denote the Legendrian figure-eight knot in $(S^3, \xi_{\rm{std}})$ with $tb(L) = -3$ (called $L_{-3}$ above).  Recall that the knot invariant $\LOSS(L)$ coming from Heegaard Floer vanishes for all Legendrian figure-eight knots, which implies that the contact invariant $EH(\xi^-(L)) = 0$ as well (see Section~\ref{sec:heegaardfloer}).

\begin{prop}\label{max tb f8}
All positive contact surgeries on $L$ are overtwisted.
\end{prop}
\begin{proof}[Sketch of Proof]
Assuming $\xi$ is tight, we will use convex surfaces to show that $(S^3 \backslash N(K), \xi^-(L))$ is contactomorphic to one of two possible contact manifolds (see Lemma~\ref{f8normalise} and Lemma~\ref{f8classification}).  We will then construct these two possibilities, and show that they have non-vanishing Heegaard Floer contact class $EH$.   However, since $\LOSS(L) = 0$, we know that $EH(\xi^-(L))$ vanishes, and so we arrive at a contradiction, and $(S^3 \backslash N(K), \xi^-(L))$ is overtwisted. We are then done, by the discussion above the proposition.
\end{proof}

Given a Seifert surface $\Sigma$ for $L$, we can think of $\Sigma$ as sitting inside $S^3 \backslash N(K)$ with boundary on $\bd \left(S^3 \backslash N(K)\right)$. After perturbing $\Sigma$ to be convex, we first wish to normalise the dividing curves of $\Sigma$ in $(S^3 \backslash N(K), \xi^-(L))$.  We will use the fact that $S^3 \backslash N(K)$ is fibred over $S^1$ with fibre $\Sigma$, and the monodromy (after choosing a basis for $\Sigma$) is given by $$\phi = \left(\begin{array}{cc}2 & 1 \\ 1 & 1\end{array}\right)$$ up to twisting along the boundary of $\Sigma$; choose the representative without any boundary twisting.

\begin{figure}[htbp]
\begin{center}
\begin{overpic}[scale=.7,tics=15]{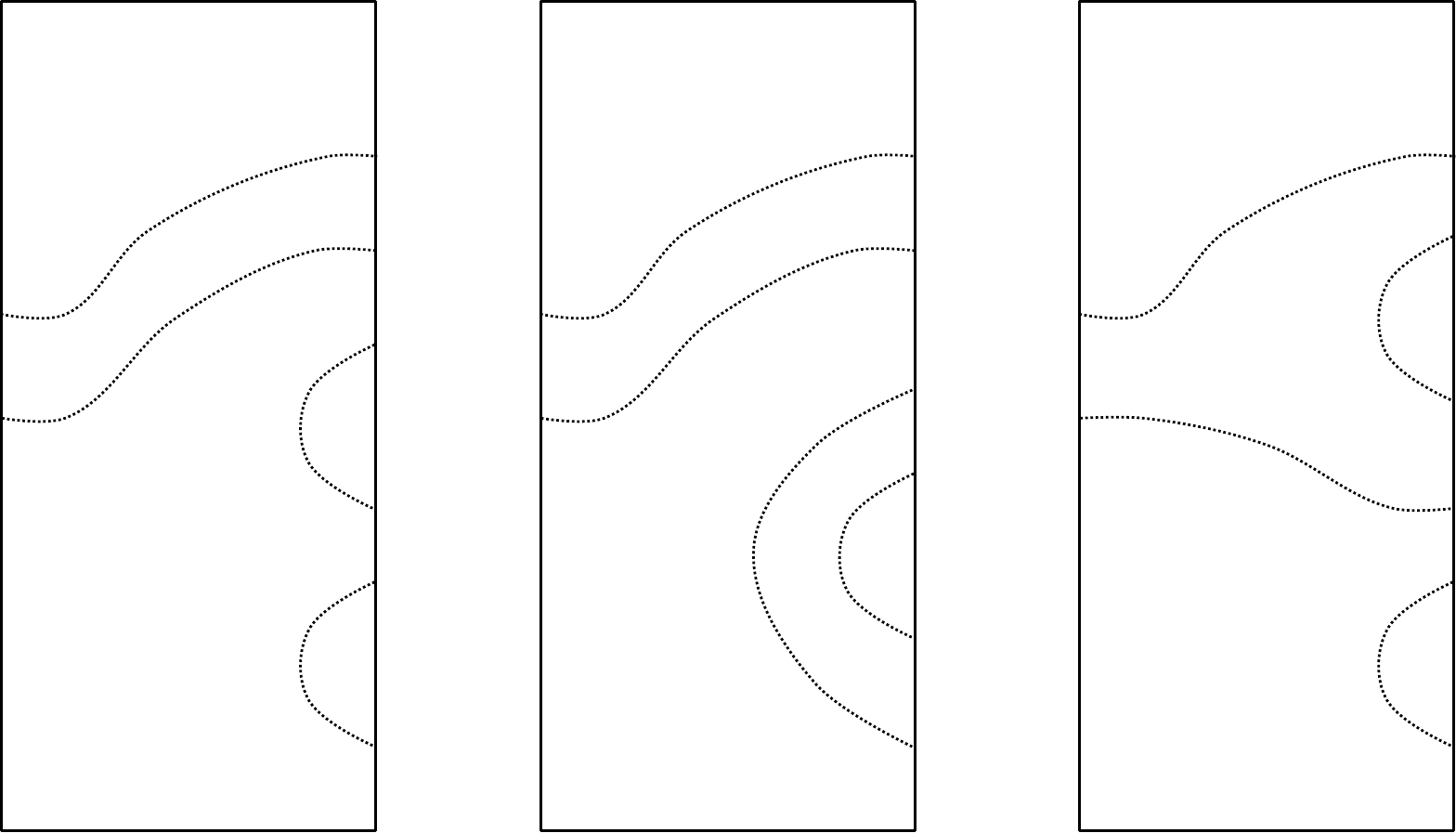}
\end{overpic}
\caption[Dividing curves extending $\Sigma$]{Possible dividing curves on the annulus $A$.  The tops are identified with the bottoms, and the left-hand side sits on $\bd \left(S^3 \backslash N(K)\right)$.}
\label{fig:dividing curves extending sigma}
\end{center}
\end{figure}

\begin{lem}\label{f8normalise} If $(S^3 \backslash N(K), \xi^-(L))$ is tight, there is an isotopic copy of $\Sigma$ in $(S^3 \backslash N(K), \xi^-(L))$ such that it is convex and the dividing curves are either
\begin{enumerate}
\item one arc and one closed curve, parallel to $\vects{0}{1}$, or
\item one boundary-parallel arc.
\end{enumerate}
\end{lem}
\begin{proof}
During this proof, we will perturb $\Sigma$ and swing it around the fibration to get new surfaces isotopic to $\Sigma$; we will call each new copy $\Sigma$.

Etnyre and Honda showed in \cite{EH:knots} that there exists a convex copy of $\Sigma$ in the complement of $N(L)$ with dividing curves consisting of three arcs, parallel to $\vects{0}{1}$, $\vects{1}{1}$, and $\vects{1}{2}$.

After gluing on a negative basic slice to get $(S^3 \backslash N(K), \xi^-(L))$, we extend $\Sigma$ to the new boundary by gluing on an annulus $A$ whose dividing curves are of one of the forms given in Figure~\ref{fig:dividing curves extending sigma}, a translate of one of those forms ({\it ie.\ }the right-hand side endpoints are shifted up/down in the $S^1$-direction from what is shown in the figure), or the image of one of those forms in a power of a Dehn twist along the core of the annulus.  Note that we have already excluded from our list of possibilities the cases where the dividing curves on $A$ trace a boundary-parallel curve along $\bd \left(S^3 \backslash N(K)\right)$.  In these cases, the dividing curves on $\Sigma$ would consist of a boundary-parallel curve and a contractible curve.  Since we are assuming that $(S^3 \backslash N(K), \xi^-(L))$ is tight, these cases would contradict Theorem~\ref{contractible}.

In any of the remaining cases, the resulting dividing curves on $\Sigma \cup A$ (which we also call $\Sigma$) consist either of a single boundary-parallel arc or one arc and one closed curve, parallel to one of $\vects{0}{1}$, $\vects{1}{1}$, or $\vects{1}{2}$, possibly with some boundary twisting ({\it ie.\ }holonomy of the dividing curves along the annulus $A$).  This holonomy can be removed, as in Honda's classification of tight contact structures on basic slices, see \cite[Proof of Proposition~4.7]{Honda:classification1}.  In the second case, it remains to show that we can remove boundary twisting, and if the dividing curves are parallel to $\vects{1}{1}$ or $\vects{1}{2}$, we can find an isotopic copy of $\Sigma$ such that the dividing curves are parallel to $\vects{0}{1}$.

We can can swing $\Sigma$ around the $S^1$-direction of the fibration $S^3 \backslash N(K) \to S^1$ to find an isotopic copy of $\Sigma$ with dividing curves changed by $\phi$ or $\phi^{-1}$.  Since $\phi^{-1}\vects{1}{1} = \vects{0}{1}$, we can go from curves parallel to $\vects{1}{1}$ to curves parallel to $\vects{0}{1}$.

Given $\Sigma$ with dividing curves consisting of an arc and a curve parallel to $\vects{1}{2}$, swinging $\Sigma$ around the $S^1$-direction of the fibration gives an embedded $\Sigma \times [0,1]$ such that $\Sigma \times \{0\}$ is convex with dividing curves parallel to $\vects{1}{2}$ and $\Sigma \times \{1\}$ is convex with dividing curves parallel to $\phi\vects{1}{2} = \vects{4}{3}$.  Let $\gamma$ be a closed $\vects{5}{4}$-curve on $\Sigma$, and consider the annulus $A' = \gamma \times [0, 1] \subset \Sigma \times [0, 1]$.  The boundary of $A'$ intersects the dividing set of $\Sigma \times \{0\}$ twelve times, while it intersects the dividing set of $\Sigma \times \{1\}$ only twice.  Thus, after making $A'$ convex with Legendrian boundary, Theorem~\ref{imbalance principle} guarantees a bypass for $\Sigma \times \{0\}$ along $A'$.  This changes the slope of the dividing curves to $\vects{1}{1}$, by Theorem~\ref{bypass Farey graph}.  Then as above, we can find an isotopic copy of $\Sigma$ whose dividing curves are parallel to $\phi^{-1}\vects{1}{1} = \vects{0}{1}$, as desired.
\end{proof}

\begin{lem}\label{f8classification} Up to contactomorphism, there are at most two tight contact structures on $S^3 \backslash N(K)$ inducing a convex boundary with two meridional dividing curves and such that there exists a copy of $\Sigma$ with dividing curves of one of the two forms described in Lemma~\ref{f8normalise}. \end{lem}
\begin{proof}
We will show that for each of the two possible normalisations in Lemma~\ref{f8normalise}, there is a unique tight contact structure up to contactomorphism.

First, we claim we can switch the signs of the regions $\Sigma_{\pm}$ of $\Sigma$.   Indeed, since $\phi = (-id)\circ\phi\circ(-id)^{-1}$, we can apply $-id$ to $\Sigma$, which keeps the same dividing curves, but switches the signs of the regions.

Given $\Sigma$ with fixed dividing curves $\Gamma$ and signs of the regions $\Sigma \backslash \Gamma$, this uniquely determines a tight vertically-invariant contact structure on some neighbourhood $N(\Sigma)$ of $\Sigma$.  We will show that there exists a unique tight contact structure on $M \backslash N(\Sigma)$, for each of the two possible choices of $\Gamma$ on $\Sigma$.  Then, given two tight contact structures on $M$ inducing the same dividing curves on $\Sigma$ with the same signs, a contactomorphism of $N(\Sigma)$ can be extended to a contactomorphism on all of $M$.

\begin{figure}[htbp]
\begin{center}
\begin{overpic}[scale=1,tics=20]{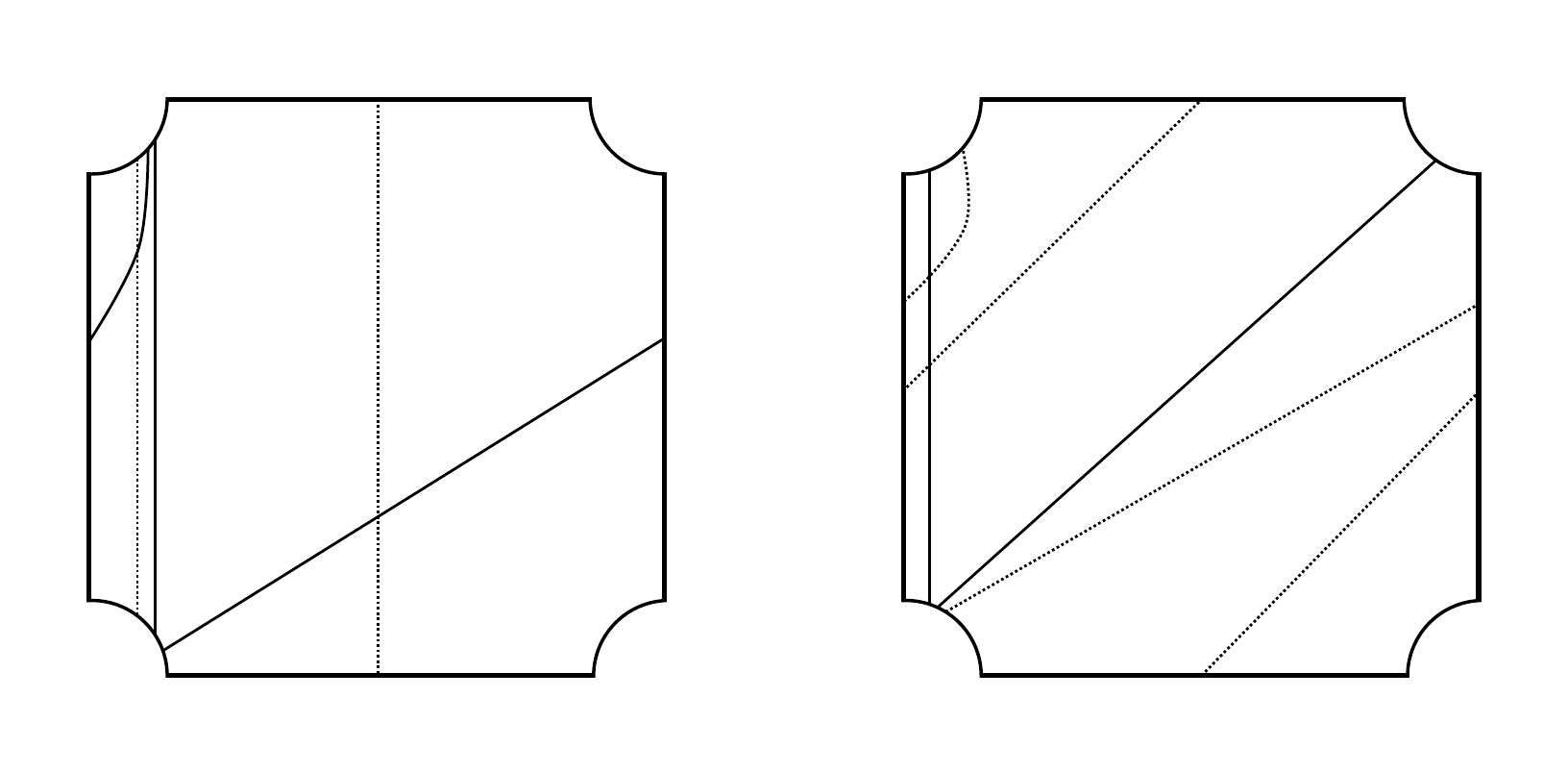}
\put(100,208){$\Sigma \times \{0\}$}
\put (340,208){$\Sigma \times \{1\}$}
\end{overpic}
\caption[$\Sigma\times\{0, 1\}$ with compressing discs, case (1)]{In each picture, the top and bottom are identified, as are the left and right sides. The dotted lines represent the dividing curves.  The solid lines represent the intersection of the boundaries $\bd D_i$ of the compressing discs with $\Sigma \times \{0, 1\}$.}
\label{fig:sigma01 with compressing discs}
\end{center}
\end{figure}

Observe that $M \backslash N(\Sigma) \cong \Sigma \times [0, 1]$ is a genus 2 handlebody. The contact structure has a convex boundary obtained by rounding the edges of $\Sigma \times \{i\}$ and $\bd \Sigma\times [0,1]$, where the dividing curves on $\Sigma \times \{0\}$ are $\Gamma$, those on $\Sigma \times \{1\}$ are $\phi(\Gamma)$, and those on $\bd \Sigma \times [0, 1]$ are two copies of $\{pt\} \times [0, 1]$.  We will look for compressing discs $D_1$ and $D_2$ such that their boundaries are Legendrian with $tb = -1$.  After making the compressing discs convex, there will be a unique choice of dividing curves for $D_i$, since their dividing curves intersect the boundary of the disc at exactly two points, by Theorem~\ref{LRP}, and there can be no contractible dividing curves, by Theorem~\ref{contractible}.  This allows us to uniquely define the tight contact structure in a neighbourhood of $\bd \left(M\backslash N(\Sigma)\right)\cup D_1 \cup D_2$.  The complement of this neighbourhood is diffeomorphic to $B^3$, and by \cite{Eliashberg:threeball}, we can uniquely extend the tight contact structure over $B^3$.

(1) {\it $\Gamma$ has one arc and one closed curve parallel to $\vects{0}{1}$:} The dividing curves on $\Sigma \times \{0, 1\}$ are shown as dotted lines in Figure~\ref{fig:sigma01 with compressing discs}.  The compressing discs are shown as solid lines.  Figure~\ref{fig:sigma01 boundary} shows the dividing curves in $\bd \Sigma \times [0, 1]$.  As the curves $\bd D_i$ pass from $\Sigma \times \{0\}$ to $\Sigma\times\{1\}$ through the region $\bd \Sigma \times [0, 1]$, they do not intersect any dividing curves, but they do switch which side of the dividing curves they are on.  Thus $\bd D_i$ intersects the dividing curves exactly twice for each $i = 0, 1$, as required.

\begin{figure}[htbp]
\begin{center}
\begin{overpic}[scale=1,tics=20]{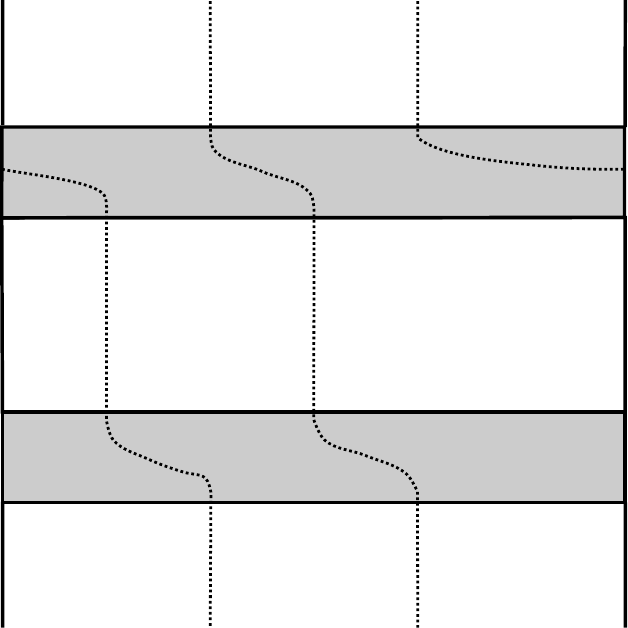}
\put(185, 160){$\Sigma \times \{0\}$}
\put(185,10){$\Sigma \times \{1\}$}
\put(185,85){$\bd\Sigma \times [0, 1]$}
\end{overpic}
\caption[$\bd\Sigma\times\lbrack0,1\rbrack$ with dividing curves]{The left and right sides are identified in this picture. The dotted lines represent the dividing curves.  The annulus in the middle is the region $\bd\Sigma \times [0, 1]$, and the darker regions above and below are interpolating regions representing how the dividing curves get connected while smoothing the boundary of $M \backslash N(\Sigma)$.}
\label{fig:sigma01 boundary}
\end{center}
\end{figure}

\begin{figure}[htbp]
\begin{center}
\begin{overpic}[scale=1,tics=20]{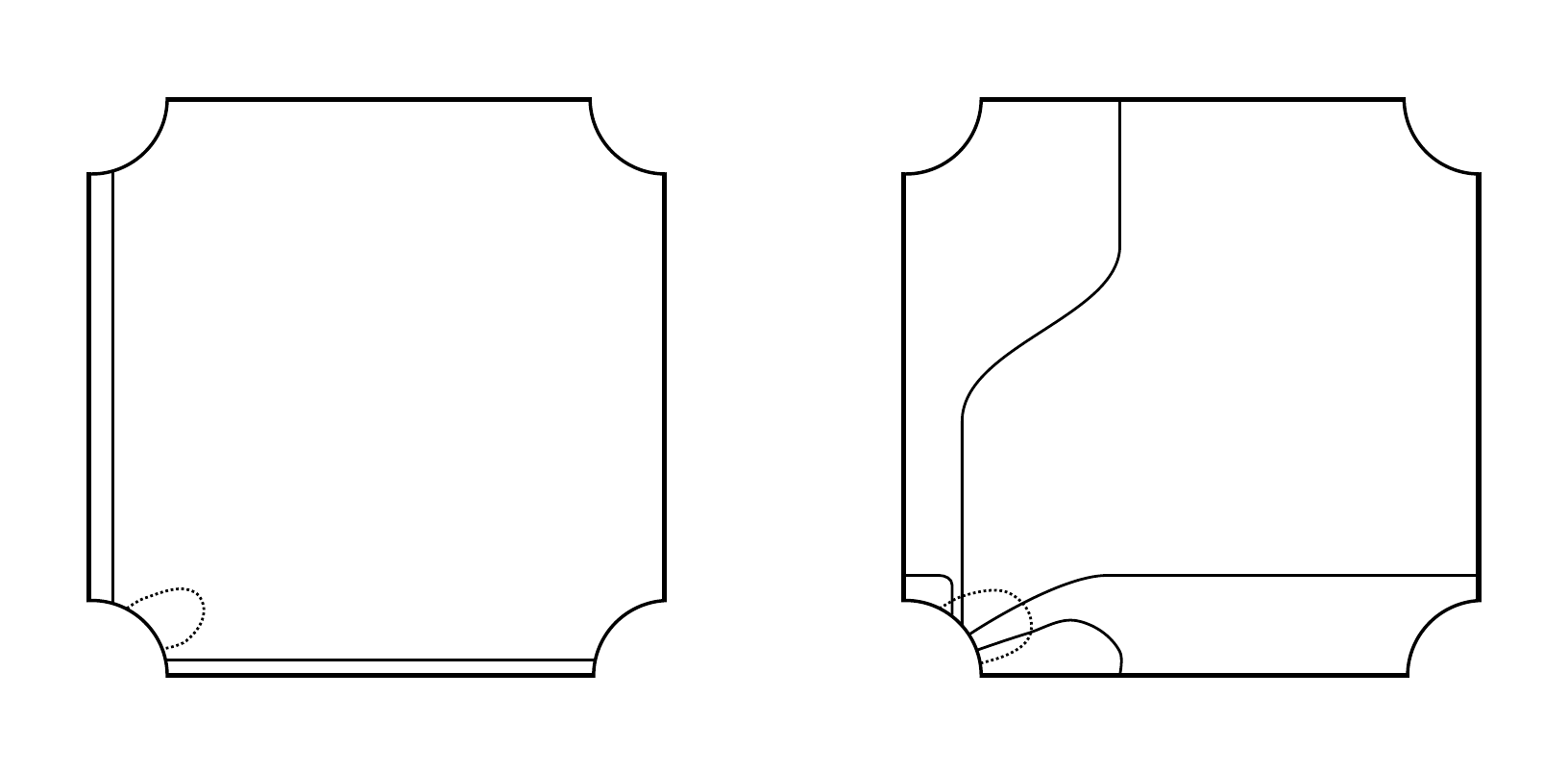}
\put(100,208){$\Sigma \times \{0\}$}
\put (340,208){$\Sigma \times \{1\}$}
\end{overpic}
\caption[$\Sigma\times\{0, 1\}$ with compressing discs, case (2)]{In each picture, the top and bottom are identified, as are the left and right sides. The dotted lines represent the dividing curves.  The solid lines represent the intersection of the boundaries $\bd D_i$ of the compressing discs with $\Sigma \times \{0, 1\}$.}
\label{fig:sigma01 bypass}
\end{center}
\end{figure}

(2) {\it $\Gamma$ has one boundary-parallel arc:} In this case, the monodromy $\phi$ does not change the dividing curves.  The dividing curves in a neighbourhood of $\bd \Sigma \times [0, 1]$ behave again as in Figure~\ref{fig:sigma01 boundary}, and the boundaries of $D_i$ intersect $\Sigma \times \{0, 1\}$ as in Figure~\ref{fig:sigma01 bypass}.  Thus $\bd D_i$ intersects the dividing curves exactly twice for each $i = 0, 1$, as required.
\end{proof}

We now exhibit both of these contact structures, and show that in each case, the Heegaard Floer contact class is non-vanishing.

\subsection*{Bypass Along $\Sigma$} Consider the open book for $S^3$ given by the figure-eight knot.  The supported contact structure on $S^3$ is overtwisted, but it was shown in \cite{EVV:torsion} that $\LOSS(L')$ is non-vanishing, where $L'$ is any Legendrian approximation of the binding of the open book.  After gluing a negative basic slice to the complement of $L'$, we arrive at $S^3 \backslash N(K)$ with a contact structure $\xi_{\rm{byp}}$ with non-vanishing Heegaard Floer contact class (since $EH(\xi_{\rm{byp}}) = \LOSS(L') \neq 0$). It it shown in \cite{EVV:torsion} that in $(S^3 \backslash N(K),\xi_{\rm{byp}})$, there exists an isotopic copy of a page of the open book (a copy of $\Sigma$) with dividing curves consisting of one boundary-parallel arc. We have thus exhibited the unique tight contact structure on $S^3 \backslash N(K)$ such that there exists a convex copy of $\Sigma$ with one boundary-parallel dividing curve arc, and it has non-vanishing $EH$ invariant.

\subsection*{No Bypass Along $\Sigma$} Consider the tight contact structure $\xi_0$ on the torus bundle $T_{\phi}$ over $S^1$ with monodromy $\phi$ that has no Giroux torsion, {\it ie.\ }created by taking a basic slice $T^2 \times [0,1]$ with dividing curves on the boundary of slopes $s_0 = -\infty$ and $s_1 = -2$ and gluing $T\times\{1\}$ to $T \times \{0\}$ via $\phi$. The contact manifold $(T_{\phi}, \xi_0)$ was shown to be Stein fillable by van Horn-Morris \cite{VHM}, so in particular, the Heegaard Floer contact class is non-vanishing, by \cite{OS:contact}.

Thinking of $T_{\phi}$ as an $S^1$-bundle over $T^2$, we pick a regular fibre and realise it as a Legendrian knot $L''$.  We claim that we can do this in a manner such that the contact planes do not twist along $L''$ when measured with respect to the fibration structure; indeed, the diffeomorphism $\phi$ is isotopic to one which fixes the neighbourhood of a point $p$ in $T^2$.  Then the knot $L'' = p \times [0, 1] \subset (T_{\phi}, \xi_0)$ is Legendrian.  By the classification in \cite[Table 2]{Honda:classification2} of tight contact structures on $T_{\phi}$, we see that in the minimally twisting one ({\it ie.\ }the one with no Giroux torsion), the contact planes twist less than an angle $\pi$ as they traverse the $S^1$-direction of the fibration.  Thus, the dividing curves on the boundary of $N(L'')$ cannot twist around the meridional direction, and so must give the product framing for the knot.  Note that $T_{\phi}\backslash N(L'')$ can be naturally identified with $S^3 \backslash N(K)$, and under this identification, the framing gives the meridional slope.

If we pick $p$ to be a point on the dividing curves of $T^2 \times \{0\} \subset T_{\phi}$, then the dividing curves on $\Sigma \subset (S^3 \backslash N(K),\xi_0|_{S^3 \backslash N(K)})$ consist of one arc and one curve parallel to $\vects{0}{1}$.  Since this embeds into $(T_{\phi}, \xi_0)$, there is a map sending $EH(\xi_0|_{S^3 \backslash N(K)})$ to $c(\xi_0)$, where the latter is non-vanishing.  This means that $EH(\xi_0|_{S^3 \backslash N(K)})$ is also non-vanishing.

\begin{remark}
Although unneeded for our proof, we can show that $\xi_{\rm{byp}}$ is not contactomorphic to $\xi_0|_{S^3 \backslash N(K)}$.  Indeed, if they were contactomorphic, then $(T_{\phi}, \xi_0)$ would be the result of some positive contact surgery on $L'$.  However, all positive contact surgeries on $L'$ are overtwisted, by \cite{Wand}.
\end{remark}

\begin{proof}[Proof of Theorem~\ref{maintheorem}] By Proposition~\ref{prop:lessthan-3} and the discussion below it, it suffices to consider the case $tb(L) = -3$.  The result of any positive contact surgery on $L$ has a contact submanifold that can be identified with $(S^3 \backslash N(K), \xi^-(L))$ or $(S^3 \backslash N(K), \xi^+(L))$.  The Heegaard Floer contact class $EH(\xi^\pm(L))$ vanishes, as $\LOSS(L) = 0$ and $L$ is amphichiral.  Since if tight, $\xi^-(L)$ and $\xi^+(L)$ would have to be contactomorphic to one of the contact structures on $S^3 \backslash N(K)$ constructed above with non-vanishing $EH$ class, we conclude that $\xi^-(L)$ and $\xi^+(L)$ are overtwisted.  Thus, any manifold which contains them as a contact submanifold must also be overtwisted.
\end{proof}

\bibliography{references}{}
\bibliographystyle{plain}
\end{document}